\documentclass[12pt]{amsart}
\usepackage{amsfonts}
\usepackage{amssymb}
\usepackage{amsmath}
\usepackage{amsthm}
\usepackage{amssymb}
\usepackage{latexsym}
\usepackage{tikz} 
\usetikzlibrary{matrix,calc,arrows,decorations.markings}
\usepackage{hyperref}

\newtheorem{thm0}{Theorem}

\newtheorem{thm}{Theorem}[section] 

\newtheorem{cor0}[thm0]{Corollary}

\theoremstyle{definition}
\newtheorem{defn}[thm]{Definition}
\newtheorem{example}[thm]{Example}

\theoremstyle{remark}
\newtheorem{remark}[thm]{Remark}

\newcommand{\Oc}{{\mathcal O}}

\newcommand{\AAA}{\mathbb{A}}
\newcommand{\ZZ}{\mathbb{Z}}

\newcommand{\QQ}{\mathbb{Q}}

\newcommand{\CC}{\mathbb{C}}

\newcommand{\PP}{\mathbb{P}}

\newcommand{\X}{X_C}
\DeclareMathOperator\PSL{PSL}
\DeclareMathOperator\CM{CM}
\DeclareMathOperator\Gr{Gr}
\DeclareMathOperator\orb{orb}
\DeclareMathOperator\red{red}
\DeclareMathOperator\rk{rk}
\DeclareMathOperator\Alb{Alb}
\DeclareMathOperator\alb{alb}
\DeclareMathOperator\Jac{Jac}
\DeclareMathOperator\Hom{Hom}
\DeclareMathOperator\im{Im}

\DeclareMathOperator\Char{Char}
\DeclareMathOperator\Pic{Pic}
\DeclareMathOperator\Sing{Sing}

\newcommand{\g}{g}

\makeatletter
\newcommand{\RightEqNo}{\let\veqno\@@eqno}
\makeatother

\newcommand\V[2]{V^{#1}_{#2}}

\title[Albanese varieties and orbifold pencils]{Albanese varieties of cyclic covers of the projective plane 
and orbifold pencils}
\author[E. Artal]{E.~Artal Bartolo}
\author[J.I. Cogolludo]{J.I.~Cogolludo-Agust{\'\i}n}
\address{Departamento de Matem\'aticas, IUMA\\ 
Universidad de Zaragoza\\ 
C.~Pedro Cerbuna 12\\ 
50009 Zaragoza, Spain} 
\email{artal@unizar.es,jicogo@unizar.es} 

\author{A.~Libgober} 
\address{ 
Department of Mathematics\\ 
University of Illinois\\ 
851 S.~Morgan Str.\\ 
Chicago, IL 60607} 
\email{libgober@uic.edu} 

\thanks{
The first two authors are partially supported. 
The first and second authors are partially supported by the Spanish Government 
MTM2016-76868-C2-2-P and by the Departamento de Industria e 
Innovación del Gobierno de Aragón and Fondo Social Europeo
\emph{E15 Grupo Consolidado Geometr\'{\i}a}.
The third author is also supported by a grant from the Simons Foundation}

\begin{document}

\begin{abstract} The paper studies a relation between 
fundamental group of the complement to a plane singular curve
and the orbifold pencils containing it.
The main tool is 
the use of
 Albanese varieties of cyclic covers ramified along such curves.
Our results give sufficient conditions for a plane singular curve to belong to an
orbifold pencil, i.e. a pencil of plane curves with multiple fibers inducing a map 
onto an orbifold curve whose orbifold fundamental group is non trivial.
We construct an example of a cyclic cover of the projective plane which is an 
abelian surface isomorphic to the Jacobian of a curve of genus 2 illustrating the 
extent to which these conditions are necessary.
\end{abstract}

\maketitle




\section*{Introduction}

There is an interesting correspondence between the fundamental groups of 
the complement to plane algebraic curves and the structure
of the pencils, possibly with multiple fibers 
which one can associate with such curves.
For example, if a plane curve $C$ is composed of a pencil, i.e.
$C=\bigcup_{i=0}^{s} C_i$ where $C_i$ are zeros of sections $t_i$
in a 2-dimensional subspace $L$ of $H^0(\PP^2,\Oc(d))$, then 
for each $P \in \X:=\PP^2\setminus C$ there is a well defined
element $t_P \in \PP(L)$ such that $t_P(P)=0$ and 
the correspondence $P \rightarrow t_P$ gives a holomorphic 
map $\X\rightarrow \PP(L)\setminus\{T_i\}_{i=0}^s$,
where $T_i$ are the points of $\PP(L)$ corresponding to the sections~$t_i$.
This map induces a surjection $\pi_1(\X) \rightarrow \pi_1(\PP(L)\setminus\{T_i\}_{i=0}^s)$ 
and hence $\pi_1(\X)$ has a free group on $s$ generators as its quotient.

In a similar vein, the existence of pencils with {\it multiple} fibers containing $C$ 
(see section~\ref{sec-orbifold}) may have implications for the fundamental group even 
if $C$ is {\it irreducible}. For example, suppose that an irreducible curve 
$C \subset \PP^2$ belongs in a pencil having two multiple fibers of multiplicities $2$ 
and $3$, i.e., the equation $F$ of $C$ can be presented as $F=f^2+g^3$ where $f,g$ are 
homogeneous polynomials.
Then the rational map $\pi: \PP^2 \dashrightarrow \PP^1$
given by $\pi([x:y:z])=[f^2:g^3]$ induces a regular map of $\X:=\PP^2\setminus C$
onto $\PP^1\setminus \{(1,-1)\}$. This map can also be viewed as an orbifold map
whose source is $\X$ with a trivial orbifold structure and whose target
is the orbifold $\CC_{2,3}$ which is an affine line with two orbifold points with 
stabilizers of orders $2$ and $3$. Such a dominant map yields a surjection
of the fundamental group $\pi_1(X)$ onto the orbifold 
fundamental group (cf. \cite{AC-prep}, \cite[Prop.2.7]{ji-libgober-mw})
for which one has 
$\pi_1^{\orb}(\CC_{2,3})=\ZZ_2*\ZZ_3$ (isomorphic to $\PSL_2(\ZZ)$).
In the rest of the paper we call a map between orbifolds having a one-dimensional 
target an {\it orbifold pencil}. The classically studied pencils (whether rational
or irrational) are a special case of orbifold pencils.

Previous work \cite{ji-libgober-mw,acl-depth,Artal-ji-Tokunaga-pencils} has 
shown that sometimes the relation between the fundamental group of a curve complement $\X$
and its {\it orbifold} pencils can be reversed, namely, the structure of the fundamental 
group provides information on the existence of (rational) orbifold pencils on $\X$ but 
the relation between fundamental groups and orbifold pencils has several aspects not 
appearing in the context of ordinary pencils. 
If a curve has only nodes and ordinary cusps as its
singularities (or more generally, singularities called in \cite{ji-libgober-mw}
$\delta$-essential) then the positivity of the rank
of the abelianization of the commutator $\pi_1(\X)'/\pi_1(\X)''$ 
implies the existence of orbifold maps on $\X$ 
(see Section~\ref{prelimsect} for more precise statements).

In the present paper we consider the correspondence between orbifold pencils
and fundamental groups of possibly reducible curves
$C$ which may have singularities much more general than ordinary 
cusps and nodes. Our main result (see Theorem~\ref{main}) describes a sufficient 
condition for the existence of orbifold pencils on $\PP^2$ containing~$C$ in terms of the 
fundamental group $\pi_1(\X)$ of its complement. Let us describe the 
results of the paper in more detail.

As in the case of curves with nodes and cusps only,
it is convenient to state our results in terms of
the Alexander invariants and the characters of the
fundamental group. The statements also use the local Albanese varieties
of singularities (cf. Section~\ref{prelimsect}). 
Recall (see more details in Section~\ref{prelimalpol}) that 
there is a notion of Alexander polynomial $\Delta_{C,\pi} \in \ZZ[t,t^{-1}]$
associated with a given surjection $\pi: \pi_1(\X)  \rightarrow \Gamma$
onto a cyclic group. 
Such a polynomial depends only on the quotient of $\pi_1(\X)$ by 
the commutator of ${\rm Ker} \pi$ and it contains information about the 
cohomology of rank one local systems on $\X$, namely, for $\chi \in \Hom(\Gamma, \CC^*)$ 
one has $H^1(\X,\chi) \ne 0$ if and only if, for a generator $\gamma$ of $\Gamma$, 
$\xi=\chi(\gamma)$ is a root of $\Delta_{C,\pi}$.
A root $\xi$ of the Alexander polynomial $\Delta_{C,\pi}$ can also be described
as an eigenvalue of the covering transformation $\tau_C$
acting on $H_1(V_C,\CC)$ where $V_C$ is a smooth model of the
cyclic cover of $\PP^2$ of degree $\deg C$ branched over $C$ (cf. \cite{li:82}).
Note that since $H^1(V_C,\CC)$ is a birational invariant, the  
eigenvalues of $\tau_C$ are independent
of a choice of the smooth model $V_C$. An alternative description 
of the multiplicity of the root $\xi$ can be given as 
the superabundance of the linear system of 
plane curves described in terms of the degree and 
the local type of the singularities of $C$. We refer to \cite{Li7} for details. 

The Alexander polynomial is affected by the local types of the
singularities of $C$ as was shown in~\cite{li:82}.
For the statement of our main results we will need 
a more precise than stated in \cite{li:82} version of 
this relation and it 
will be shown below in Section~\ref{mainproof}.

\begin{thm0}\label{singularityexist} 
Let $C$ be a plane curve with arbitrary 
singularities and let $\chi$ be a character of finite order $N>0$ 
of the fundamental group $\pi_1(\X)$. 
Assume that $\chi$ is ramified along each
irreducible component of $C$. Assume also that $H^1(\X,\chi) \ne 0$.
Then there exists a singularity $P\in C$ with local equation $f_P(x,y)=0$
for which the following property holds.
 
Denote by $B_P$ a Milnor ball about $P$  
and let $\chi_P$ be the character 
of $\pi_1(B_P\setminus C)$
which is the composition
$$\pi_1(B_P  \setminus C) \rightarrow \pi_1(\X)
\overset{\chi}{\rightarrow} \CC^*$$
where the left map is induced by the 
inclusion $B_P  \setminus C\hookrightarrow \X$.
Then the corresponding map:
\begin{equation*}\label{surjectivitylocal}\tag{T1}
H^1(\X,\chi) \rightarrow H^1(B_P\setminus C,\chi_P)
\end{equation*}
has a non-trivial image (in particular $H^1(B_P\setminus C,\chi_P) \ne 0$).
\end{thm0}

The orbifold pencils on $\PP^2$ which we attach to the curve $C$ 
are obtained from irrational pencils on $V_C$ and are 
constructed using the Albanese map $V_C \rightarrow \Alb(V_C)$.
Albanese varieties of cyclic covers of $\PP^2$ were 
considered classically for covers of small degree 
(cf.~\cite{comessatti,defranchis,catanese} for a modern exposition).
The work of Comessatti~\cite{comessatti,catanese} studies
the irregular $3$-cyclic coverings of the plane, and he finds examples
both for Albanese dimensions~$1$ and $2$. In the latter case,
he constructs an example (also found in \cite{defranchis}
and thoroughly explained in \cite{catanese}) such that
the cyclic cover is the product of two copies of a special elliptic curve.
Bagnera and deFranchis take another viewpoint: they study rational cyclic
quotients of abelian surfaces. However, as presented in Theorem~\ref{singularityexist}, 
our focus is on $C$ and its algebraic/topological properties such as cohomology 
conditions on its complement.

Note that we obtain an explicit model
of such quotients in Theorem~\ref{thm-degt}; the ramification curve is 
described and we derived geometric properties of this curve from this fact. 
Our construction depends on the relation between $\Alb(V_C)$ and 
the invariants of singularities of $C$ described in  
\cite{cm}. There are also simple cases where 
such irrational orbifold pencils come up in a straightforward way. 
This is the case when $\Alb(V_C)$ is an elliptic curve, or analogously, for 
curves whose local Alexander polynomial equals $t^2-t+1$. 
More generally we have the following,

\begin{cor0}[cf. {Theorem \ref{main}}]\label{corollarymain}
Let $C$, $\chi$ be as in Theorem{\rm~\ref{singularityexist}}
and let $\V{\chi}{C}$ be a smooth projective model of the cyclic cover 
associated with the kernel of $\chi$.
Assume that the Albanese dimension of $\V{\chi}{C}$ is equal to one
(see Theorem{\rm~\ref{oka}} for explicit examples).
Then $C$ is an element of a  
global quotient orbifold pencil such that $\chi$ is the pullback 
of a character of the orbifold fundamental group of the target of this 
orbifold pencil. 
\end{cor0}

We want to  relax the assumption on Albanese dimension in the  Corollary{\rm~\ref{corollarymain}}
and assume only that {\it one has a one-dimensional image in one of the isogeny 
$\chi$-equivariant factor of $\Alb(V_C)$}.
In what follows, we will describe how, under some restriction on the analytic type of the singularities of~$C$, 
we may identify the abelian varieties which are the isogeny $\chi$-equivariant factors of $\Alb(V_C)$ 
projection onto which may lead to construction of an orbifold pencil.

This restriction on the analytic type of singularities is 
given in terms of introduced in \cite{ji-libgober-mw} the local Albanese varieties 
associated with a plane curve singularities (cf. section
\ref{prelimsect}  for definition.)
A local Albanese variety is equipped with an automorphism
i.e. a $\ZZ$-action 
coming from the action of the semi-simple part of the local monodromy
on the homology of the Milnor fiber.
The relation between local Albanese varieties of singularities and 
global information about $C$ comes from canonical maps
of each local Albanese variety into $\Alb(V_C)$.  
The sum of these maps over all singularities of $C$  
surjects onto $\Alb(V_C)$ (cf. \cite{cm}).
These maps from the local Albanese varieties of the singularities 
of $C$ are $\ZZ$-equivariant with respect to the just mentioned 
monodromy action and the action of the (cyclic) covering group of~$V_C$.


Before stating the main result of this paper (Theorem~\ref{main}) 
we shall state sufficient conditions for existence of orbifold pencil 
in the case when singularties of $C$ have type $\AAA_{p-1}$ 
and 
for which fewer technical assumptions can be made.

\begin{thm0}\label{mainspecialcase}  
Let $C$, $\chi$, $P$, and $\V{\chi}{C}$ be as in 
Theorem{\rm~\ref{singularityexist}} and Corollary{\rm~\ref{corollarymain}} above.
Assume that $C$ has at $P$ an $\AAA_{p-1}$-singularity, $p$ an odd prime, and in particular, 
the local Albanese variety $\Alb_P$ is the Jacobian of the curve $D$ of genus~$g:=\frac{p-1}{2}$.
Let $\alb_{{\chi},D}: \V{\chi}{C} \rightarrow \Jac(D)$ be the composition 
of the Albanese map $\V{\chi}{C} \rightarrow \Alb(\V{\chi}{C})$ with the projection 
on its isogeny component~$\Jac(D)$. 

\smallskip 
If the image of $\alb_{{\chi},D}$ has dimension one, then there is a pencil
$\V{\chi}{C} \rightarrow D$ inducing an orbifold pencil
\begin{equation*}\tag{T3}\label{mainspecialcasepencil}
     \X \dasharrow D/\im\chi
\end{equation*}
onto the global quotient of $D$ by the canonical action of $\im \chi$ on~$D$. 

Moreover the character $\chi$ is the pullback on $\pi_1(\X)$ of a 
character of $\pi_1^{\orb}(D/\im\chi)$ via the pencil~\eqref{mainspecialcasepencil}.
\end{thm0}

Now we are ready to state the main result of the paper 
with milder than in Theorem \ref{mainspecialcase} restriction 
on singularities of $C$ but similar conclusion that global 
orbifold pencils exist.

\begin{thm0}\label{main}
Let $C$, $\chi$, $N$ and $P$  be as in Theorem{\rm~\ref{singularityexist}}.
Let $\V{\chi}{C}$ be a smooth projective model of the cyclic
branched cover of $\PP^2$
associated with the kernel of $\chi$ and let $\tau_C^{\chi}$ be the map
induced by the deck transformation on $H_1(\V{\chi}{C},\CC)$.


\begin{enumerate}
\item\label{componentj}
Assume that the local Albanese variety $\Alb_P$ of the singularity $P$
has an isogeny component $J_{\chi}$ satisfying the following:
\begin{enumerate}
\makeatletter\renewcommand{\p@enumii}{}\makeatother 
\item\label{componentjchi}
The action of $\im \chi$ on $\Alb_P$ induces an action on 
$J_{\chi}$ and the map $J_{\chi} \rightarrow \Alb(\V{\chi}{C})$ 
induced by the $(\im \chi)$-equivariant map $\Alb_P \rightarrow \Alb(\V{\chi}{C})$ 
has a finite kernel. 
 \item\label{jacobianassumption}
$J_{\chi}$ is the Jacobian of a curve $D$ such that $D$ is a quotient 
of an exceptional curve $\mathcal D$ of positive genus in a resolution 
of the singularity $z^N=f_P(x,y)$ i.e. $D=\mathcal D/\Delta({\mathcal D},\chi)$ 
where $\Delta({\mathcal D},\chi)\subseteq \im \chi$ 
is a (possibly trivial) subgroup of the covering group $\im \chi$,
the latter being considered as an automorphism group of~$\mathcal D$.
\end{enumerate}
 Let $\alb_{{\chi},D}$ be the composition of the Albanese map
$\V{\chi}{C} \rightarrow \Alb(\V{\chi}{C})$ with the projection
on the factor $J_{\chi}=\Jac(D)$. If the dimension of the image of
$\alb_{{\chi},D}$ is one, then there exists a pencil
\begin{equation*}
   \V{\chi}{C} \rightarrow D
\end{equation*}
inducing an orbifold pencil 
\begin{equation*}
\label{mainorbpencil}\tag{T4}
\X \dasharrow D^{\orb}_{\im \chi}
\end{equation*}
where $D^{\orb}_{\im \chi}=D/({\im \chi}/\Delta(D,\chi))$ is the 
global quotient orbifold obtained 
via the induced action of $({\im \chi}/\Delta(D,\chi))$ on $D$.
For such an orbifold pencil \eqref{mainorbpencil} the character~$\chi$ 
is the pull-back on $\pi_1(\X)$ of a character
of $\pi_1^{\orb}(D^{\orb}_{\im \chi})$ via~\eqref{mainorbpencil}.

\item\label{simple}
If $\Alb_P$ 
is simple (i.e. is not isogenous 
to a product of abelian varieties of positive dimension) 
then the assumptions \eqref{componentjchi} and \eqref{jacobianassumption} in 
\eqref{componentj} are automatically satisfied.

\end{enumerate}
\end{thm0}

Note that assumption \eqref{componentj}\eqref{componentjchi} means that 
$J_{\chi}$ is an $({\im \chi})$-equivariant isogeny
component of $\Alb(\V{\chi}{C})$. In particular it implies that 
the tangent space to $J_{\chi}$ at the identity is 
contained in the $\chi$-eigenspace
of $\tau _C^{\chi}$ acting on the tangent space of $\Alb(\V{\chi}{C})$
at the identity,

The conditions for existence of orbifold pencils given by this 
theorem have the following converse showing 
that the existence of an orbifold pencil
having the curve $C$ as a member, implies that the Albanese variety
of the corresponding cyclic cover splits up to isogeny.
Some factors of this splitting are the Jacobians of the curves 
with the orbifold associated with the pencils being the global quotients 
of these curves.

More precisely (see section~\ref{sec-orbifold} for definitions related
to orbifold pencils) one has:

\begin{thm0}\label{main2}
Suppose that $C$ belongs to a global quotient orbifold pencil $\pi$
(cf. Definition{\rm~\ref{orbpencil}})
of target $\PP^1$ with orbifold points of multiplicities ${\bar m}=(m_1,\dots,m_s)$
so that $\pi$ induces a homomorphism 
$\pi_1(\X) \rightarrow \pi_1^{\orb}(\PP^1_{\bar m})$.
Assume also that there is $\rho \in \Char \pi_1^{\orb}(\PP^1_{\bar m})$
such that $\chi=\pi^*(\rho)$ and also that
the orbifold $\PP^1_{\bar m}$ is a global quotient of a curve $\Sigma$.
Then $\Alb(\V{\chi}{C})$ admits an $(\im\chi)$-equivariant
surjection onto~$\Jac(\Sigma)$
and hence one has an $(\im\chi)$-equivariant
isogeny $\Alb(\V{\chi}{C})\sim \Jac(\Sigma) \times A$
for an abelian $(\im\chi)$-variety~$A$.

More generally, if there is a finite number $\phi_1,\dots,\phi_n$ of global
quotient orbifold pencils as above
with targets $(\PP^1_{\bar m},\rho)$ ($\rho \in {\rm Char}
\pi_1^{\orb}(\PP^1_{\bar m})$
which are \emph{$\QQ$-strongly independent},
then $\Alb(\V{\chi}{C})$ admits an
$(\im\chi)$-surjection onto~$\Jac(\Sigma)^n$, that is there is an equivariant
isogeny, $\Alb(\V{\chi}{C})\sim \Jac(\Sigma)^n \times A$ for
an abelian $(\im\chi)$-variety~$A$.


\end{thm0}

The proofs of Theorems~\ref{singularityexist},~\ref{mainspecialcase},~\ref{main} 
and Theorem~\ref{main2} are presented in Section~\ref{mainproof}.
In Section~\ref{examples} we consider applications
of Theorem~\ref{main}. Firstly we discuss an example
of a curve $C$ with $\AAA_{2g}$-singularities, i.e. whose singularities are
locally isomorphic to $u^2+v^{2g+1}=0$, which belongs to an orbifold pencil.
For the curves described in Theorem~\ref{oka},
all roots of the Alexander polynomial correspond to 
orbifold pencils on the complement. The Albanese variety
of the canonical cyclic cover $\V{}{C}$ is the Jacobian
of a certain curve of genus~$g$ (described as a Belyi cover).
In Theorem~\ref{thm-degt} we give an example of a curve
for which the Albanese variety is the same as one of those in Theorem~\ref{oka}
(for the particular case of $g=2$), but whose characters corresponding to the
roots of the Alexander polynomial 
cannot be obtained as pull-back via orbifold maps.
The difference between the curves in Theorems~\ref{oka} and~\ref{thm-degt} 
comes from the difference in the Albanese maps of the corresponding
cyclic covers, namely, the images of the Albanese maps have different dimensions.
The curve given in explicit way 
described in Theorem~\ref{thm-degt} is particularly interesting, since its
canonical cyclic cover has as a minimal model an abelian surface
(specifically the Jacobian of a curve of genus~2 cf. 
also \cite{catanese}).
This construction of an abelian surface via cyclic coverings
branched over curves given by explicit equation 
can be of independent interest.
Finally, in Theorem~\ref{ji} we present a family of curves contained in more than 
one orbifold pencil and for which the Albanese dimension is maximal, that is, two.

\numberwithin{equation}{section}

\begin{section}{Preliminaries}\label{prelimsect}

\subsection{Alexander polynomials}\label{prelimalpol}(cf. \cite{li:82})

Let $C$ be a plane curve with irreducible components $C_0,C_1,\dots,C_r$
where $F_i(x,y,z)=0$ is a reduced equation of~$C_i$ of degree~$d_i$. 
Then $H_1(\X,\ZZ)$ is an abelian group of
rank~$r$ isomorphic to $\ZZ^{r+1}/{(d_0,\dots,d_{r})\ZZ}$.
This isomorphism is given by 
$$
\gamma \mapsto\left(\frac{1}{2 \pi \sqrt{-1}}\int_\gamma \frac{dF_i}{F_i}\right)_{i=0}^r.
$$
Fix a surjection $\pi: \pi_1(\X) \rightarrow \Gamma$
onto a cyclic group $\Gamma$. Note that $\pi$ can be factored 
through $H_1(\X,\ZZ)$ and hence induces a homomorphism 
$\ZZ^{r+1}/{(d_0,\dots,d_{r})\ZZ}\to \Gamma$.
Let $K=\ker \pi$. Consider the exact sequence 
$0 \rightarrow K/K'\rightarrow \pi_1(\X)/K' \rightarrow \Gamma \rightarrow 0$ and
the corresponding action of $\Gamma$ on $K/K' \otimes \CC$.
The \emph{Alexander polynomial} $\Delta_{C,\pi}(t)$ of $C$
(relative to the surjection $\pi$)
is the characteristic polynomial associated with the action of $\Gamma$ 
on the vector space $K/K' \otimes \CC$. Note that $\dim K/K' \otimes \CC
<\infty$ (cf. \cite{li:82}), $\Delta_{C,\pi}$ has
integer coefficients and in the case of irreducible
$C$ is independent, for all the previous choices, as an element in $\CC[t,t^{-1}]$ modulo
units. In the latter case, if $\Gamma=\ZZ/d\ZZ$, then $K/K'$ is the abelianization of the
commutator of~$\pi_1(\X)$.

Zeroes of the Alexander polynomial can be described in terms
of the cohomology of local systems as follows. Note that, since
$\pi_1(\X)/K'$ is abelian, $\pi$ factors through a character, say $\chi$.
Let $\xi \in \CC^*, (1,...,1) \in
\ZZ^{r+1}/{(d_0,\dots,d_{r})\ZZ}=H_1(\X,\ZZ)$ be
a generator of $\im\chi \in \CC^*$; one has:
\begin{equation}\label{rootslocalsyst}
\Delta_{C,\pi}(\xi)=0  \Longleftrightarrow
\dim H^1(\X,\chi)>0\quad
\text{(cf. \cite{eh:97,li:01}).}
\end{equation} 

The Alexander polynomial is restricted by the local type of singularities
and the degree of $C$ as follows. Each singularity $P \in C$,
has associated its local Alexander polynomial $\Delta_{C}^{P}$, or equivalently
the characteristic polynomial of the local monodromy acting on
the Milnor fiber of the singularity (cf. \cite{mil}).
Then one has the divisibility relation (cf. \cite{li:82})
\begin{equation}\label{divisibility}
    \Delta_C(t) \vert \Pi_P \Delta^P_C(t).
\end{equation}
Moreover the roots of the Alexander polynomial are roots of
unity of the degree $\deg C$.

\begin{example}\label{2pexample} Let $C$ be a curve whose singularities
are topologically equivalent to the $\AAA_{2g}$-singularity 
with local equation $u^2=v^{2g+1}$.
Since the characteristic polynomial of the monodromy
for such singularity is $ \frac{t^{2g+1}+1}{t+1}$
the Alexander polynomial of $C$ is trivial unless $2({2g+1}) \vert \deg C$ and
moreover it is equal to $\left(\frac{t^{2g+1}+1}{t+1}\right)^s$
for some $s \ge 0$.
\end{example}

\bigskip
\noindent

\subsection{\texorpdfstring{Local Albanese Varieties and                        
singularities of $\CM$-type}{Local Albanese Varieties and                       
singularities of CM-type}}\label{localalbsubsect}
\mbox{}

Let $f=0$ be a germ of an isolated (i.e. reduced) plane curve
singularity at the origin. Let $M_f$ be the Milnor fiber of $f$, i.e. the
intersection of a sufficiently small ball $B_{\epsilon}$ about the origin
and the hypersurface $f=t, 0<|t| \ll \epsilon$. The cohomology of
$M_f$ (more generally, the cohomology of the Milnor fiber
of an isolated hypersurface singularity)
supports {\it the limit} mixed Hodge structure.  It was constructed by
Steenbrink and we refer to  \cite{st:77} for its study. Here 
we only note that it depends on the {\it family}  of 
germs $f=t$, rather than its specific member and  
record the following properties
of this mixed Hodge structure used below:

\begin{enumerate}
\item
It has weight $2$ and the weight filtration is associated with the
unipotent part of the monodromy $T_u$ in the decomposition
into unipotent and semisimple parts of $T=T_sT_u$, the monodromy 
operator acting on~$H^1(M_f,\CC)$.

\item
The size of Jordan blocks of the monodromy operator is
at most 2 and equals $\rk W_0$. Moreover,
$$\rk\Gr^W_2 - \rk W_0=r-1,$$
where $r$ is the number of branches of $f=0$.

\item
The Hodge filtration is invariant under the action of
the semisimple part of the monodromy. Note that by the Monodromy
Theorem the order of $T_s$ is finite (cf. \cite{st:77}).

\item
Let $L_f$ be the link of the singularity $z^n=f$ where
$n$ is the order of the automorphism $T_s$.
Then
\begin{equation}\label{uptocover}
\Gr^W_1H^1(M_f)=
\Gr_3^WH^2(L_f)(1)
\end{equation}
(where $H(1)$ is the Tate twist of a Hodge structure $H$, cf. \cite[Proposition~3.1]{cm}).
\end{enumerate}

\begin{defn} The \emph{local Albanese variety} of the germ $f$ is defined as the abelian
variety $(\Gr_F^0H^1(M_f))^*/H_1((M_f,\ZZ)$, with polarization
induced by the intersection form on $H_1(M_f,\ZZ)$. Equivalently,
the local Albanese is the abelian part of the semiabelian
variety associated by Deligne (cf. \cite{deligne:74}) to
the 1-motif in the case of the mixed Hodge structure
dual to the limit mixed Hodge structure discussed above.
\end{defn}

The above definition is rather technical but it admits a simpler description
in terms of the resolution of the singularity $z^n=f$ discussed above. Let $\tilde B
\rightarrow B$ be an embedded log-resolution of the germ $f=0$,
$V_n \rightarrow B$ be the projection of the germ of the singularity
$z^n=f$ onto $B$. The singularities of the normalization
of the fiber product $S=V_n \times_B \tilde B$ are cyclic
quotient singularities and their (minimal) resolution
$\widetilde {S}$ provides a
resolution of the singularity of the germ $z^n=f$ (cf.
\cite{lipman:74}). In this resolution the boundary of the tubular
neighborhood of the exceptional locus can be identified with
the link $L_f$ in (\ref{uptocover}). Moreover $H^1(L_f)$
(and by duality $H^2(L_f)$) can be identified in an appropriate
way with $\bigoplus H^1(E_i)$ where $E_i$ runs through the
set of exceptional curves in $\widetilde {S}$
having a {\it positive genus}. More precisely, we have:

\begin{thm}\label{productofjactheorem}
{\rm{(cf.~\cite[Theorem 3.11]{cm}\rm)}}
Let $f(x,y)=0$ be a singularity with a semi-simple
monodromy and let $N$ be the order of the monodromy operator.
The Albanese variety of the germ $f(x,y)=0$
is isogenous to a product of Jacobians of the exceptional
curves of positive genus for a resolution of
\begin{equation}\label{cycliccoversing}
z^N=f(x,y).
\end{equation}
\end{thm}

The latter description suggests an approach to defining the
local Albanese for the non-reduced case, i.e. as
the product of the Jacobians of curves of positive genus
in the resolution of the singularities of the germ~$z^n=f$.

Finally recall the following:

\begin{defn}(cf.~\cite[Definition~3.4]{cm}) 
 A plane curve singularity \emph{has a $\CM$-type} if its local 
 Albanese variety is an abelian variety of $\CM$-type.
A plane curve singularity \emph{has a $\CM$-type} if its local 
Albanese variety is isogenous to a product of simple abelian varieties of $\CM$-type.
\end{defn}

We refer to~\cite{Shimura} for basic information regarding
abelian varieties of $\CM$-type.
Unibranched singularities and singularities for
which the characteristic polynomial of the monodromy operator
has no multiple roots provide many examples of singularities
of $\CM$-type (cf.~\cite{cm}).

\begin{example}\label{ex:belyi}
Let $(C,P)$ be a simple curve singularity of type $\mathbb{A}_{2 g}$, with local equation
$y^2-x^{2 g+1}=0$. The local Albanese variety is associated to the surface singularity
$y^2-x^{2 g+1}=z^{2(2g+1)}$. For any resolution of 
this surface singularity, there is only one non-rational irreducible component 
$D_{\mathbb{A}_{2 g}}$ of its exceptional divisor,
which is a Belyi cover of the unique branching component of the minimal resolution of $(C,P)$,
ramified at the three intersection points with the other components, with ramification indices
$2,2 g+1, 2(2 g+1)$, whose genus is~$g$.
\end{example}

\subsection{Orbifold Pencils}\label{sec-orbifold}
\begin{defn} Let $X$ be a quasi-projective manifold and $S$ be an orbicurve (one-dimensional orbifold). 
A holomorphic map $\phi$ between $X$ and the underlying $S$ complex curve 
we shall call an 
\emph{orbifold pencil} if the index of each orbifold point $p$ divides the multiplicity of each connected 
component of the fiber $\phi^*(p)$ over~$p$.
\end{defn}
We will concentrate our attention on orbifold pencils of curve complements. Let $C\subset \PP^2$ be a plane 
curve (not necessarily irreducible) and let $\X$ denote its complement. Consider an orbifold pencil 
$\phi:\X\to S$, where $S$ is a rational orbifold curve (that is, its compactification is $\PP^1$)
given by a finite number of orbifold points, say $P_i$, $i=1,\dots,s$, with orbifold structure of order
$m_i\in \ZZ_{>0}\cup \{\infty\}$, $i=1,\dots,s$ 
(i.e. near which the orbifold chart is the chart given by a disk with 
the standard action of the cyclic group of order $m_i$). 
For convenience, $m_i=\infty$ means that $P_i$ has been
removed from $S$, namely, $S=\PP^1\setminus \{P_i\mid m_i=\infty\}$. In the future we will denote $S$ 
simply by~$\PP^1_{\bar m}$, where~$\bar m:=(m_1,\dots,m_s)$.

\begin{defn}\label{orbpencil}
In the situation as above, we say that $C$ \emph{belongs to an orbifold} pencil of type $\bar m$. 
Moreover, the orbifold pencil $\phi$ will be called
 a \emph{global quotient orbifold pencil} if there exists a morphism 
$\Phi: X_G \rightarrow \Sigma$, where $X_G$ is a quasi-projective manifold endowed with
an action of a finite group $G$ and $\Sigma$ a curve which makes the diagram
\begin{equation}\label{diagramorbchar}
\begin{tikzpicture}[description/.style={fill=white,inner sep=2pt},baseline=(current bounding box.center)]
\matrix (m) [matrix of math nodes, row sep=2.5em,
column sep=2.5em, text height=1.5ex, text depth=0.25ex]
{X_G & \Sigma \\
\X& \PP^1_{\bar m}\\};
\path[->,>=angle 90](m-1-1) edge node[auto] {$\Phi$} (m-1-2);
\path[->,>=angle 90](m-2-1) edge node[auto] {$\phi$} (m-2-2);
\path[->,>=angle 90](m-1-1)edge node[auto,swap] {}(m-2-1);
\path[->,>=angle 90](m-1-2)edge node[auto,swap] {}(m-2-2);
\end{tikzpicture}
\end{equation}
commutative, for which the vertical arrows are the models for the quotients by the action of $G$.

If in addition, there is a character $\chi\in \Char(\X)$ and a
character $\rho\in \Char^{\orb}(\PP^1_{\bar m})$ such that $\chi=\rho\circ \pi$,
and $X_G$ (resp. $\Sigma$) is the covering of $\X$ (resp. $\PP^1_{\bar m}$) associated
with the character $\chi$ (resp. $\rho$), then we say $(C,\chi)$ \emph{belongs to a global
quotient orbifold pencil with target $(\PP^1_{\bar m},\rho)$}.
\end{defn}

Oftentimes, the set of global quotient orbifold pencils --up to the obvious equivalence
by automorphisms of the target-- is infinite (see~\cite{ji-libgober-mw,acl-depth}). A very useful (cf. Theorem \ref{ji} below)
property to determine the different nature of such orbifold pencils is given by the
following.

\begin{defn}
\label{def-indep}
Global quotient orbifold pencils $\phi_i:(\X,\chi) \rightarrow (\PP^1_{\bar m},\rho)$, $i=1,\dots,n$
are called \emph{independent} if the induced maps $\Phi_i: X_G \rightarrow \Sigma$
define $\ZZ[G]$-independent morphisms of modules
\begin{equation}
\label{eq-independent}
{\Phi_i}_*: H_1(X_G,\ZZ) \rightarrow H_1(\Sigma,\ZZ),
\end{equation}
that is, independent elements of the $\ZZ[G]$-module $\Hom_{\ZZ[G]}(H_1(X_G,\ZZ),H_1(\Sigma,\ZZ))$.

In addition, if 
\begin{equation}
\label{eq-strongly}
\bigoplus {\Phi_i}_*: H_1(X_G,\ZZ) \rightarrow H_1(\Sigma,\ZZ)^n
\end{equation}
is surjective we say that
the pencils $\phi_i$ are \emph{strongly independent}. 
If the previous morphism \eqref{eq-strongly} is considered with 
coefficients over $\QQ$, then we will use the term $\QQ$-strongly independent.
\end{defn}
\end{section}

\begin{section}{Proof of theorems \ref{singularityexist},
\ref{mainspecialcase}, 
\ref{main} and \ref{main2}}\label{mainproof}

\begin{proof}[Proof of Theorem {\rm~\ref{singularityexist}}]
We shall use notations set up in the Introduction and in Section~\ref{prelimsect} 
and consider the Alexander polynomial $\Delta_{C,\chi}(t)$
of $C$ relative
to the homomorphism $\chi: \pi_1(\X) \rightarrow \Gamma \subset \CC^*$
where $\Gamma=\im \chi$ is the group of $N$-th roots of unity by hypothesis. 
Let $\xi \in \CC^*$ be a primitive $N$-th root of unity. 
Since $H^1(\X,\chi)\ne 0$ one has
$\Delta_{C,\chi}(\xi)=0$ (cf.~\eqref{rootslocalsyst}).
Let $S_\chi:=\{P\in\Sing(C)\mid H^1(B_P\setminus C,\chi_P)\ne 0\}$;
because of~\eqref{divisibility}, this set is non-empty.

For each $P\in S_\chi$,
consider the unbranched covering
of $E_P:=B_P\setminus C$ corresponding to the surjection
$\pi_1(E_P) \rightarrow \Gamma$ and denote it by
 $(\widetilde{E}_P)_{\Gamma}$. 
Then the restriction
of the cyclic cover of $B_P$ given by the equation
\eqref{cycliccoversing} on
$E_P$
is equivalent to $(\widetilde{E}_P)_{\Gamma}
\rightarrow E_P$. 
The proof of the Divisibility Theorem 
(cf.~\cite{li:82,cm}) also shows that that there is a surjection
$\bigoplus_{P\in S_\chi} H_1((\widetilde{E}_P)_{\Gamma},\QQ)_\xi
\rightarrow H_1(V_C^\chi, \QQ)_\xi$, where the subindex $\xi$ 
stands for the $\xi$-eigenspace of the corresponding deck transformations.
Hence one can take as $P$ in \eqref{singularityexist}
any singular point in $S_\chi$ for which the map 
$H_1((\widetilde{E}_P)_{\Gamma},\QQ)_\xi
\rightarrow H_1(V_C^\chi, \QQ)_\xi$ has a non trivial image.
%
\end{proof}

\begin{remark}
In fact $H^1(B_P\setminus C,\chi_P)\ne 0$ is not enough to 
ensure that the map~\eqref{surjectivitylocal} in Theorem~\ref{singularityexist} has a non-trivial image.
For instance, consider $C$ a sextic curve
with seven ordinary cusps. It is well known 
(already to O.Zariski, cf.~\cite{Artal-couples,li:82,ji-libgober-mw} 
for more recent discussions) that there is a conic passing 
through six out of the seven cusps. The Alexander polynomial of $C$ is $t^2-t+1$, 
which coincides with the local Alexander polynomials of its singularities.
However, if $\chi$ is a character of order 6, the map 
$$
H^1(\X,\chi) \rightarrow H^1(B_P\setminus C,\chi_P)
$$
is not trivial if and only if $P$ is one of the six cusps on the conic.
\end{remark}

\begin{proof}[Proof of Theorem {\rm~\ref{main}}]
Now let us assume that $P$ is a singularity satisfying 
Theorem~\ref{main}
and consider a resolution of the associated surface singularity $V_P=\{z^N=f_P(x,y)\}$,
where $N$ is the order of the character~$\chi$ 
and $f_P$ is a local equation of $C$ near $P$.
Recall that such a resolution can be obtained (Jung's method 
cf. \cite{lipman:74}) by normalizing 
a pull-back of an embedded resolution of the singularity at~$P$.
 
It follows from the A'Campo formula (cf. \cite{ac:75}),
or from discussion in Section~\ref{localalbsubsect},
that $\xi$ is the root of
the characteristic polynomial of the transformation
induced on homology by the action $z \mapsto \xi z$ 
on a resolution of singularities of the surface $V_P$
and restricted 
to one of the curves
of positive genus in the resolution of the singularity~$V_P$.
Denote such a curve by~$\mathcal D$. 
Jung's procedure implies that , 
$\mathcal D$ is an irreducible component of a $\Gamma$-cover of a rational
curve (namely an exceptional divisor of the resolution of~$P$).
By Theorem~\ref{productofjactheorem} (i.e.
\cite[Theorem 3.11]{cm}) there is an isogeny 
component of the Jacobian of $\mathcal D$ (possibly a direct 
sum of several simple components) 
which is also an isogeny component
of $\Alb(\V{\chi}{C})$. If this component is an $(\im \chi)$-invariant Jacobian
of a curve $D$, i.e. if the assumption~\eqref{jacobianassumption} 
in Theorem~\ref{main} 
is fulfilled, then by Torelli's Theorem $\im \chi$ acts 
on $D$ as well 
(unfaithfully if $\mathcal D \ne D$). 
Note that Theorem~\ref{main} allows
non-reduced curves $f=0$, which are excluded in the statement of
Theorem~\ref{productofjactheorem}. 

As a consequence of Jung's method, the resolutions of $z^n=f$ and $z^n=f_{\red}$, where $f_{\red}$ is 
the product of irreducible factors of $f$, are both obtained by pull-back and normalization
of the same embedded resolution
of the curve $f_{\red}=0$. 
In particular the conclusions of Theorem~\ref{main}\eqref{componentj} still hold in the non-reduced 
case, whereas $D$ depends on the ramification data of the cyclic cover~$V_P$.

Returning to the proof of the existence 
of an orbifold pencil satisfying~\eqref{componentj}, 
suppose that the composition of the Albanese map and the projection
onto $\Jac(D)$ has a 1-dimensional image~$W$. Let $\sigma:\mathcal{D}\to D$ be the quotient map. Consider the diagram
\begin{equation}
\begin{tikzpicture}[description/.style={fill=white,inner sep=2pt},baseline=(current bounding box.center)]
\matrix (m) [matrix of math nodes, row sep=2.5em,
column sep=5em, text height=1.5ex, text depth=0.25ex]
{&&D&\mbox{}\\
\mathcal{D} & \Jac(\mathcal{D}) &\Jac(D)&\hspace*{-3.5cm} =J_\chi\\
\V{\chi}{C} & \Alb(\V{\chi}{C})& W&\hspace*{-3cm}:=\im \alb_{{\chi},D}.\\};
\path[->>,>=angle 90](m-2-1) edge node[auto] {$\sigma$} (m-1-3);
\path[right hook->,>=angle 90](m-2-1) edge node[auto] {} (m-2-2);
\path[->>,>=angle 90](m-2-2) edge node[auto] {$\Jac(\sigma)\ $} (m-2-3);
\path[right hook->,>=angle 90](m-1-3) edge node[auto] {} (m-2-3);
\path[left hook->,>=angle 90](m-2-1) edge node[auto] {} (m-3-1);
\path[->,>=angle 90](m-2-2) edge node[auto,swap] {$\neq 0$} (m-3-2);
\path[right hook->,>=angle 90](m-3-1)edge node[auto,swap] {}(m-3-2);
\path[->>,>=angle 90](m-3-2)edge node[auto,swap] {}(m-2-3);
\path[->>,>=angle 90](m-3-2)edge node[auto,swap] {}(m-3-3);
\path[->,>=angle 90, bend right= 60,dashed](m-3-4)+(-1.5,0.3) edge(m-1-3);
\path[right hook->,>=angle 90](m-3-3)edge node[auto,swap] {}(m-2-3);
\end{tikzpicture}
\end{equation}
This diagram shows that the image of $\mathcal D$ in $\Jac(D)$ 
coincides with the image of $D$ and hence it is contained in~$W$. 
The assumption that $\dim \im \alb_{{\chi},D}=1$ hence 
yields that $\im \alb_{\chi}=D$ (up to a translation).
Moreover the map $\V{\chi}{C} \rightarrow D$ is $\Gamma$-equivariant
and hence it induces the orbifold pencil as described in 
Theorem~\ref{main}\eqref{componentj}.

If $\Jac(D)$ is a simple abelian variety, then $J_{\chi}=\Jac(D)$ 
as it follows from the discussion above. This yields~\eqref{simple}
which concludes the proof of Theorem~\ref{main}.
\end{proof}

\begin{proof}[Proof of Theorem {\rm~\ref{mainspecialcase}}] 
To derive this proof from Theorem~\ref{main} 
we have to verify that its hypotheses are satisfied.
For $\AAA_{2g}$-singularities, one has $\Alb_P=\Jac(\mathcal{D})$,
where $\mathcal{D}$ is a covering of the branching component
of the minimal resolution of the singularity. 
Note that under the hypothesis $p=2 g+1$ is prime, 
$\Jac(\mathcal{D})$ is simple (cf. \cite[Example~4.8(1)]{Shimura}).
Using~Theorem~\ref{main}\eqref{simple}, the result follows.
%
\end{proof}

\begin{proof}[Proof of Theorem {\rm \ref{main2}}] Recall that
$\pi_1^{\orb}(\PP^1_{\bar m})=
\pi_1(\PP^1\setminus \{P_i\}_{i=1}^s)/\langle\gamma_i^{m_i}\rangle_{i=1}^s$
where $\gamma_i$ are meridians about 
the points $P_i$ in $\pi_1(\PP^1\setminus \{P_i\}_{i=1}^s)$.
Consider the composition $\lambda_{\rho}$ 
$$\pi_1(\PP^1\setminus \{P_i\}_{i=1}^s)
\overset{\lambda}{\longrightarrow} 
\pi_1^{\orb}(\PP^1_{\bar m})
\overset{\rho}{\longrightarrow} \CC^*.$$
Following the notation introduced in Definition~\ref{orbpencil},
consider the natural surjection morphism
$\Lambda:\pi_1(\PP^2\setminus (C \cup \bigcup_{i=1}^{s} D_i))
\rightarrow \pi_1(\X)$.
Note that the meridians about the components $D_i$ generate the normal subgroup $\ker \Lambda$.
Since they are taken by $\pi$ onto $m_i$-th powers of (eventually powers of) meridians about $P_i$, 
the surjection $\pi$ is induced by  
$\pi_1(\PP^2 \setminus (C \cup \bigcup_{i=1}^{s} D_i)) \rightarrow
\pi_1(\PP^1 \setminus  \{P_i\}_{i=1}^{s})$. Hence we have the following
commutative diagram:
\begin{equation}\label{diagramproof12}
\begin{tikzpicture}[description/.style={fill=white,inner sep=2pt},baseline=(current bounding box.center)]
\matrix (m) [matrix of math nodes, row sep=2.5em,
column sep=2.5em, text height=1.5ex, text depth=0.25ex]
{\pi_1(\PP^2 \setminus (C \cup \bigcup_{i=1}^{s} D_i)) & \pi_1(\PP^1\setminus \{P_i\}_{i=1}^{s})
 \\
\pi_1(\X)  & \pi_1^{\orb}(\PP^1_{\bar m})\\};
\path[->,>=angle 90](m-1-1) edge node[auto] {$\Pi$} (m-1-2);
\path[->,>=angle 90](m-2-1) edge node[auto] {$\pi$} (m-2-2);
\path[->,>=angle 90](m-1-1)edge node[auto,swap] {$\Lambda$}(m-2-1);
\path[->,>=angle 90](m-1-2)edge node[auto,swap] {$\lambda$}(m-2-2);
\end{tikzpicture}
\end{equation}

Since $\chi=\pi^*(\rho)$, the character $\chi$ is the
composition $\pi_1(\X) \overset{\pi}{\longrightarrow} \pi_1^{\orb}(\PP^1_{\bar m})
\overset{\rho}{\longrightarrow} \CC^*$, one has
$\Pi(\ker (\Lambda \circ \chi)) \subseteq \ker (\lambda \circ \rho)$.
Hence diagram (\ref{diagramproof12}) shows that $\pi$ induces
the map of covering spaces
\begin{equation}\label{unramifiedmap}
\left(\PP^2 \setminus (C \cup \bigcup_{i=1}^{s} D_i)\right)_{\Lambda
\circ \chi} \longrightarrow
\left(\PP^1\setminus \{P_i\}_{i=1}^{s}\right)_{\lambda_{\rho}}
\end{equation}
corresponding to the subgroups $K:=\ker(\Lambda \circ \chi)$ and
$K_{\rho}:=\ker(\lambda_{\rho})$ respectively.
The extension of the map (\ref{unramifiedmap}) to a smooth compactification
of $\left(\PP^2 \setminus (C \cup \bigcup_{i=1}^{s} D_i)\right)_{\Lambda
\circ \chi}$ and then to a resolution of its base points yields a map of a 
birational model of $\V{\chi}{C}$ to $\Sigma$; recall that the orbifold
$\PP^1_{\bar m}$ is a global quotient of a Riemann surface~$\Sigma$. And hence
we have also a map 
$\Alb(V^{\chi}_C) \rightarrow \Jac(\Sigma)$.
The Poincar{\'e} Reducibility Theorem yields an isogeny between
$\Alb(V^{\chi}_C)$ and $\Jac(\Sigma) \times A$.

In the case of $n>1$ pencils $\phi_1,\dots,\phi_n$, we
obtain a corresponding map for each $\phi_i$
and hence a map $\Alb(\V{\chi}{C}) \rightarrow \Jac(\Sigma)^n$.
By Definition~\ref{def-indep},
the corresponding map of $H^1$ is surjective and hence, as above,
the Poincar{\'e} Reducibility Theorem yields the claimed isogeny.
\end{proof}

\end{section}

\begin{section}{\texorpdfstring{Curves with $\AAA_{2\g}$-singularities}{Curves with A{2g}-singularities}}\label{examples}

The purpose of this section is to justify the lengthy statements of the main theorems
by highlighting both their power and their subtleties through a series of examples.
Simplifying the statements would only cause a more coarse description of the actual connection 
between characteristic varieties and orbifold pencils.

In what follows we present three essentially different types of situations: the pivotal example 
is shown in Theorem~\ref{thm-degt}, where $\dim \im \alb=2$, $\Alb(\V{\chi_2}{C_2})$ is a simple abelian 
variety, which is the Jacobian of a curve, and hence the image $\alb$ projected onto the isogeny
factors of $\Alb(\V{\chi_2}{C_2})$ is never a curve. Therefore the conditions of Theorem~\ref{main} are
not satisfied. Moreover, $(C_2,\chi_2)$ does not contain a global orbifold pencil (see~\cite{AC-prep}).
Another remarkable fact is that $\V{\chi_2}{C_2}$ is birational to an abelian 
surface of $\CM$-type corresponding to the cyclotomic field~$\QQ(\zeta_5)$. The cyclic quotients of
these abelian surfaces have been studied by Bagnera and deFranchis~\cite{defranchis}; this curve
is the ramification divisor of one of such quotients.

On the other hand in Theorem~\ref{oka} a curve $C_1$ is exhibited (for $k=1$ and $g=2$) whose 
$\Alb(\V{\chi_1}{C_1})$ coincides with $\Alb(\V{\chi_2}{C_2})$, however $\dim \im \alb=1$, which implies, 
by Theorem~\ref{mainspecialcase}, the existence of a global orbifold pencil containing $(C_1,\chi_1)$.
Finally, in Theorem~\ref{ji}, $\dim \im \alb=2$, as for $C_2$. However $\Alb(\V{\chi_3}{C_3})$ 
decomposes (up to isogeny) as a product of three simple Jacobians of curves and the image $\alb$ projected 
onto these factors are always 1-dimensional. By Theorem~\ref{mainspecialcase} this implies the existence
of three independent global orbifold pencils containing~$(C_3,\chi_3)$.


\begin{thm}\label{oka}
Let $C_1$ be an irreducible curve in $\PP^2$ given by the equation
\begin{equation}\label{okacurve}
f_{2k}^{2\g+1}+f_{(2\g+1)k}^2=0,
\end{equation}
where $f_i$ is a generic homogeneous polynomial of degree~$i$.
Let $\chi_1$ be the character
of $\pi_1(\PP^2\setminus C_1)$ sending the generator
of $H_1( \PP^2\setminus C_1)=\ZZ_{2k(2\g+1)}$ to a primitive root
of unity of degree $2(2\g+1)$.
Consider $\V{\chi_1}{C_1}$ the cyclic covering of order~$2(2\g+1)$ of $\PP^2$ ramified along~$C_1$.
Let $D_{\AAA_{2\g}}$ be the curve of genus $\g$ which is the cyclic
Belyi cover of $\PP^1_{(2,2g+1,2(2g+1))}$ of degree~$2(2\g+1)$.
Then $\Alb(\V{\chi_1}{C_1}) \sim \Jac(D_{\AAA_{2\g}})$ and the Albanese dimension of $\V{\chi_1}{C_1}$ is~1.
\end{thm}

\begin{remark}
These curves were studied by M.~Oka in \cite{oka:75} and the pencil provided by Theorem~\ref{mainspecialcase}
is the one generated by $f_{2k}^{2\g+1}$ and $f_{(2\g+1)k}^2$. Also note that 
$\Jac(D_{\AAA_{2\g}})$ is the local Albanese variety of any singularity of~$C_1$, see Example~\ref{ex:belyi}.
\end{remark}

\begin{proof} The curve (\ref{okacurve}) has $2k^2(2g+1)$ singularities
each locally equivalent to $u^2=v^{2\g+1}$ forming scheme theoretical
(for generic  $f_{2k},f_{(2\g+1)k}$) complete
intersection $\mathcal B$ given by $f_{2k}=f_{(2\g+1)k}=0$. The Example~\ref{2pexample} provides
a general form of its Alexander polynomial and a calculation
using \cite{Li7} shows that $s=1$ i.e. it is
$ \frac{t^{2 \g+1}+1}{t+1} $.

Consider the pencil of curves of degree $2k(2\g+1)$ given by:
\begin{equation}\label{okapencil}
   \pi_{C_1}: [x_0:x_1:x_2] \mapsto [f_{2k}(x_0,x_1,x_2)^{2\g+1}:f_{2\g+1}(x_0,x_1,x_2)^{2k}]
\end{equation}
yielding a regular map $\PP^2\setminus \mathcal B \rightarrow \PP^1$.
We shall view this as an orbifold pencil with target $\PP^1_{2,{2g+1}}$. Since $\pi_{C_1}(C_1)=p\in \PP^1$,
this map induces another orbifold map
$\PP^2\setminus C_1 \rightarrow \PP^1_{2,{2g+1}}\setminus \{p\}$ by restriction. Note that the inclusion
$\PP^1_{2,{2g+1}}\setminus \{p\} \hookrightarrow \PP^1_{2,2\g+1,2(2g+1)}$ is a dominant
map. The latter orbifold is a global orbifold
quotient by the action of cyclic group $\ZZ_{2(2\g+1)}$
of a cyclic Belyi cover $\Sigma$ having genus $g$ (the value of the genus
follows for example from the Riemann-Hurwitz formula).
Moreover the pencil (\ref{okapencil}) lifts to the regular
map $\tilde \pi_{C_1}: \V{\chi_1}{C_1} \rightarrow \Sigma$. It follows from
\cite{li:82} that $\dim H_1(\V{\chi_1}{C_1})=2g$. Hence the
induced map $\Pi_{C_1}: \Alb(\V{\chi_1}{C_1}) \rightarrow \Jac(\Sigma)$
is an isogeny and one has the commutative diagram:
\begin{equation}
\begin{matrix} \V{}{C_1}& \overset{\tilde \pi_{C_1}}{\longrightarrow} & \Sigma & \cr
                \downarrow & & \downarrow &  \cr
                \Alb(\V{\chi_1}{C_1})& \overset{\tilde \Pi_{C_1}}{\rightarrow}  & \Jac(\Sigma) & \cr
\end{matrix}
\end{equation}
where the vertical arrows are the Albanese map and the canonical embedding of $\Sigma$ into
its Jacobian. This implies that the Albanese image of $\V{\chi_1}{C_1}$ is one dimensional.
\end{proof}

\begin{thm}\label{thm-degt}
Let $C_2$ be the union of a self-dual quintic $C_0$ with 3 $\AAA_4$-singularities
and the line $L$ which is tangent to $C_0$ at one of its singularities, say $P^0$. Consider $\chi_2$ any character
of order 10 that ramifies along $C_0+5L$ (the coefficients represent the ramification indices).
Then

\begin{enumerate}
 \item\label{thm-degt-1}
 The canonical class of the minimal model of $\V{\chi_2}{C_2}$ is trivial.
 \item\label{thm-degt-2}
 $\dim H_1(\V{\chi_2}{C_2},\CC)=4$. In particular this minimal model is an abelian surface.
 \item\label{thm-degt-3}
 This abelian surface is isomorphic to the $\Jac(D_{\AAA_4})$ which is 
 a simple abelian variety and hence its 
 Albanese dimension is~2.
\end{enumerate}
\end{thm}

\begin{proof}
In order to prove part~\eqref{thm-degt-1}, we will construct the 10th-cyclic cover
of $\PP^2$ associated with $\chi_2$.
Note that $K_{\PP^2}=-3H=-\frac{3}{5}C_0$. Denote by $\hat \PP^2$ the resulting surface
(see Figure~\ref{fig-hatP2}) after blowing up the singular points of $C_0$ to obtain a normal
crossing divisor and then blowing down the preimage of $L$.

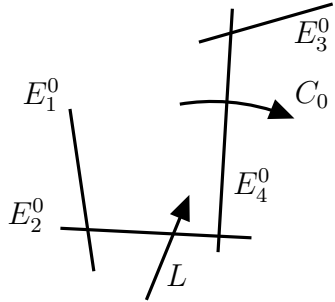
\begin{figure}[ht]
\begin{center}
\begin{tikzpicture}[scale=.4,y=-1cm]
\draw[very thick,black] (10.63625,7.9375) -- (16.98625,8.255);
\node (e20) at (9.5,7.5) {$E_2^0$};
\draw[very thick,black] (15.875,8.73125) -- (16.35125,0.635);
\node (e40) at (17,6.5) {$E_4^0$};
\draw[very thick,black] (15.24,1.74625) -- (19.685,0.47625);
\node (e03) at (19,1.5) {$E_3^0$};
\draw[very thick,black] (10.95375,3.96875) -- (11.7475,9.36625);
\node (e01) at (10,3.5) {$E_1^0$};
\draw[very thick,arrows=-triangle 45,black] (14.605,3.81) ..controls (15.5575,3.65125) and (16.8275,3.65125) .. (18.415,4.28625);
\node (c0) at (19,3.5) {$C_0$};
\draw[very thick,arrows=-triangle 45,black] (13.49375,10.31875) -- (14.9225,6.82625);
\node (l) at (14.5,9.5) {$L$};
\end{tikzpicture}%
\caption{Local resolution at $P^0$}
\label{fig-local-resol}
\end{center}
\end{figure}

To understand this, we will briefly describe the local resolution of the singularity at $P^0$
shown in Figure~\ref{fig-local-resol}. The subindices of $E_i^0$ indicate the order of appearance
of the exceptional divisors. Since the first two blow-ups occur on infinitely near smooth points
of $L$, its self-intersection drops by~$2$. However, these first two infinitely near points are
not smooth on $C_0$, but of multiplicity 2. Since two more blow-ups on infinitely near smooth
points of $C_0$ are required to resolve the singularity, the self-intersection of $C_0$ drops
by~$2\cdot (2)^2+2\cdot (1)^2=10$.

We denote by $P^{\pm}$ the other singular points of type $\AAA_4$. Note that Figure~\ref{fig-local-resol}
(excluding the germ of $L$) also describes a resolution of $P^{\pm}$ in $C_0$. For the corresponding
exceptional divisors we replace the superscript $0$ by $\pm$ accordingly. Analogously as mentioned above,
the self-intersection of $C_0$ drops by $10$ at each point.

By B{\'e}zout's Theorem, $L$ intersects $C_0$ at another point. Since its self-intersection after the
blow-ups is $-1$ and it intersects only $C_0$ and $E_2^0$, we can blow it down keeping the normal crossing
property. The self-intersection of both $E_2^0$ and $C_0$ increases by~1.
The resulting surface is $\hat\PP^2$ and the involved divisors are shown in Figure~\ref{fig-hatP2}.
By the Projection Formula we obtain
$$
K_{\hat \PP^2}=-\frac{3}{5}C_0-\frac{1}{5}\left(E_1^++E_1^-+E_1^0+2E_2^++2E_2^-+2 E_2^0\right).
$$
\begin{figure}[ht]
\begin{center}
\begin{tikzpicture}[scale=.4,y=-1cm]
\node (e20) at (13,7.4) {$E_2^0$};
\node[below=-4] at (e20.south) {$(-2)$};
\draw[very thick,black] (10.63625,7.9375) -- (16.98625,8.255);
\node (e40) at (17,4) {$E_4^0$};
\draw[very thick,black] (15.875,8.73125) -- (16.35125,0.635);
\node (e30) at (19,1.5) {$E_3^0$};
\draw[very thick,black] (15.24,1.74625) -- (19.685,0.47625);
\node (e10) at (10,4) {$E_1^0$};
\draw[very thick,black] (10.95375,3.96875) -- (11.7475,9.36625);
\node (e2+) at (24,16.5) {$E_2^+$};
\node[below=-2] at (e2+.south) {$(-3)$};
\draw[very thick,black] (27.46375,17.145) -- (21.11375,17.4625);
\node (e4+) at (22.8,13) {$E_4^+$};
\node[left=-5] at (e4+.west) {$(-1)$};
\draw[very thick,black] (22.225,17.93875) -- (21.74875,9.8425);
\node (e3+) at (20,9.4) {$E_3^+$};
\node[below=-1] at (e3+.south) {$(-2)$};
\draw[very thick,black] (22.86,10.95375) -- (18.415,9.68375);
\node (e1+) at (28,14) {$E_1^+$};
\node[left=-5] at (e1+.west) {$(-2)$};
\draw[very thick,black] (27.14625,13.17625) -- (26.3525,18.57375);
\node (e2-) at (3,16.5) {$E_2^-$};
\draw[very thick,black] (0.3175,17.145) -- (6.6675,17.4625);
\node (e4-) at (6.5,15) {$E_4^-$};
\draw[very thick,black] (5.55625,17.93875) -- (6.0325,9.8425);
\node (e3-) at (8,9.3) {$E_3^-$};
\draw[very thick,black] (4.92125,10.95375) -- (9.36625,9.68375);
\node (e1-) at (1.7,14) {$E_1^-$};
\draw[very thick,black] (0.635,13.17625) -- (1.42875,18.57375);
\node (c0) at (13,13) {$C_0(-5)$};
\draw[very thick,black] plot [smooth] coordinates {(4.60375,13.0175)  (7.14375,13.335)  (9.8425,13.335) (11.90625,12.22375)  (13.49375,10.31875)   (14.12875,8.89)   (15.08125,6.985)  (16.8275,6.0325)  (19.3675,6.6675)   (22.70125,8.5725)   (24.4475,10.795)  (24.4475,12.7)   (23.495,13.97)  (21.43125,14.76375)   (20.32,14.76375)};
\end{tikzpicture}%
\caption{Surface $\hat \PP^2$}
\label{fig-hatP2}
\end{center}
\end{figure}
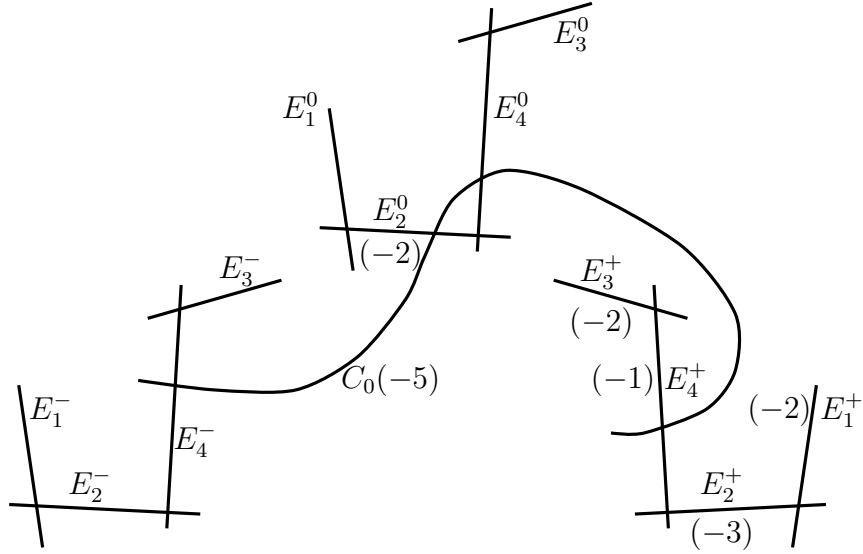

The self-intersections of the divisors are shown in parenthesis unless $(E_i^{\bullet})^2=(E_i^{+})^2$.
Since we have blown-up 12 points and blown-down one exceptional divisor, one can compute the Euler
characteristic as follows:
$$
\chi (\hat \PP^2)=\chi (\PP^2)+12-1=14.
$$
An alternative way to obtain a surface birationally equivalent to $\V{\chi_2}{C_2}$ is to consider the
10th cyclic cover of $\hat \PP^2$ ramified along the total transform of $C_0+5L$, that is,
$$
\array{rl}
R:=C_0+7E^0_1+14E_2^0+15E^0_3+30E_4^0+2E_1^{\pm}+4E_2^{\pm}+5E_3^{\pm}+10E_4^{\pm}&\\
\equiv C_0+7E^0_1+4E_2^0+5E^0_3+2E_1^{\pm}+4E_2^{\pm}+5E_3^{\pm} & \mod 10\Pic(\hat\PP^2),
\endarray
$$
where $E_i^{\pm}=E_i^{+}+E_i^{-}$.
It is easier to factor such covering as the composition of a double cover $\pi_2$ and a 5th-fold cover~$\pi_5$.

The double cover of $\hat\PP^2$ is ramified along
$$R_2:=C_0+E_3^{\pm}+E_3^{0}+E_1^{0}\equiv R \quad \mod 2\Pic(\hat \PP^2)$$
and will be denoted by $X$. The dual graph of the total transform $\pi_2^*(R)$ in $X$ is shown in~Figure~\ref{fig-X}.

\begin{figure}[ht]
\begin{center}
%
\begin{tikzpicture}[scale=.4,y=-1cm, vertice/.style={draw,circle,fill,minimum size=0.25cm,inner sep=0}]

\node[vertice] (e40) at (12.7,15.24) {};
\node[above left] (e40l) at (e40.north) {$e_4^0$};
\node[below] (e40e) at (e40.south) {$(-2)$};
\node[vertice] (e20) at (15.24,15.24) {};
\node[above right] (e20l) at (e20.north) {$e_2^0$};
\node[below] (e20e) at (e20.south) {$(-4)$};
\node[vertice] (e30) at (10.16,15.24) {};
\node[above] (e30l) at (e30.north) {$e_3^0$};
\node[below] (e30e) at (e30.south) {$(-1)$};
\node[vertice] (e10) at (17.78,15.24) {};
\node[above] (e10l) at (e10.north) {$e_1^0$};
\node[below] (e10e) at (e10.south) {$(-1)$};
\node[vertice] (c0) at (13.97,10.16) {};
\node[above] (c0l) at (c0.north) {$c_0$};
\node[right] (c0e) at (c0l) {$(-2)$};
\node[vertice] (e1++) at (22.5425,7.62) {};
\node[above] (e1++l) at (e1++.north) {$e_1^{++}$};
\node[right] (e1++e) at (e1++.east) {$(-2)$};
\node[vertice] (e2++) at (20.0025,7.62) {};
\node[above] (e2++l) at (e2++.north) {$e_2^{++}$};
\node[left] (e2++e) at (e2++.west) {$\ (-3)$};
\node[vertice] (e4+) at (17.4625,10.16) {};
\node[above] (e4+l) at (e4+.north) {$e_4^{+}$};
\node[below] (e4+e) at (e4+.south) {$(-2)\ $};
\node[vertice] (e3+) at (21.2725,10.16) {};
\node[above] (e3+l) at (e3+.north) {$e_3^{+}$};
\node[below] (e3+e) at (e3+.south) {$(-1)$};
\node[vertice] (e2+-) at (20.0025,12.7) {};
\node[below] (e2+-l) at (e2+-.south) {$e_2^{+-}$};
\node[vertice] (e1+-) at (22.5425,12.7) {};
\node[below] (e1+-l) at (e1+-.south) {$e_1^{+-}$};
\node[vertice] (e1--) at (5.3975,7.62) {};
\node[left] (e1--l) at (e1--.west) {$e_1^{--}$};
\node[vertice] (e2--) at (7.9375,7.62) {};
\node[above] (e2--l) at (e2--.north) {$e_2^{--}$};
\node[vertice] (e4-) at (10.4775,10.16) {};
\node[above] (e4-l) at (e4-.north) {$e_4^{-}$};
\node[vertice] (e2-+) at (7.9375,12.7) {};
\node[below] (e2-+l) at (e2-+.south) {$e_2^{-+}$};
\node[vertice] (e1-+) at (5.3975,12.7) {};
\node[left] (e1-+l) at (e1-+.west) {$e_1^{-+}$};
\node[vertice] (e3-) at (6.6675,10.16) {};
\node[left] (e3-l) at (e3-.west) {$e_3^{-}$};
\draw[very thick] (c0) -- (e40);
\draw[very thick] (c0) -- (e20);
\draw[very thick] (e40) -- (e30);
\draw[very thick] (e10) -- (e20);
\draw[very thick] (e4-) -- (c0);
\draw[very thick] (c0) -- (e4+);
\draw[very thick] (e1++) -- (e2++);
\draw[very thick] (e2++) -- (e4+);
\draw[very thick] (e1+-) -- (e2+-);
\draw[very thick] (e2+-) -- (e4+);
\draw[very thick] (e3+) -- (e4+);
\draw[very thick] (e1--) -- (e2--);
\draw[very thick] (e2--) -- (e4-);
\draw[very thick] (e1-+) -- (e2-+);
\draw[very thick] (e2-+) -- (e4-);
\draw[very thick] (e3-) -- (e4-);
\draw[very thick,style={bend right}] (e40) to (e20);
\draw[very thick,style={bend left}] (e40) to(e20);
\end{tikzpicture}%
\caption{Surface $X$}
\label{fig-X}
\end{center}
\end{figure}
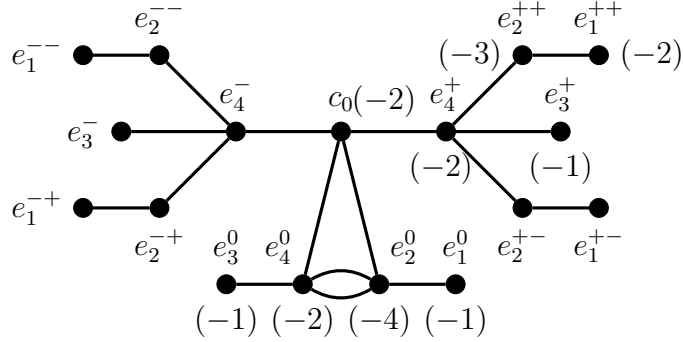

In order to compute the self-intersection of each divisor one has to apply the intersection theory formulas
for covers (cf.~\cite[Chapter~II. Section~10]{barth-compact}).

Note that
$$K_X=-\frac{1}{5}c_0+\frac{3}{5}e_1^0-\frac{2}{5}e_2^0+e_3^0-\frac{1}{5}e_1^{\pm\pm}-\frac{2}{5}e_2^{\pm\pm}+e_3^{\pm},$$
where $e_i^{\pm\pm}$ denotes the sum $e_i^{++}+e_i^{+-}+e_i^{-+}+e_i^{--}$.

By Riemann-Hurwitz, the Euler characteristic of $X$ can be obtained as
$$
\chi(X)=2(\chi(\hat\PP^2)-\chi(R_2))+\chi(\pi_2^*(R_2))=2(14-10)+10=18.
$$

After blowing down the divisors $e_1^0$, $e_3^0$, $e_3^{+}$, and $e_3^{-}$ one obtains the surface $Y$, where
$$K_Y=-\frac{1}{5}\left(c_0+2e_2^0+e_1^{\pm\pm}+2e_2^{\pm\pm}\right) \quad 
\text{ and } \quad \chi(Y)=14.$$

Finally one needs to perform the 5:1 cover of $Y$ ramified along $R_5:=c_0+2e_2^0+e_1^{\pm\pm}+2e_2^{\pm\pm}$,
which incidentally is the support of~$K_Y$.
Note that this divisor has 5 connected components, namely, $e_1^{++}+2e_2^{++}$, $e_1^{+-}+2e_2^{+-}$,
$e_1^{-+}+2e_2^{-+}$, $e_1^{--}+2e_2^{--}$, and $c_0+2e_2^0$, each with the same combinatorial structure
as shown at the bottom of Figure~\ref{fig-Y}. The appropriate ramified cover on $e_1^{++}+2e_2^{++}$ is 
shown in Figure~\ref{fig-Y}. Next to each irreducible component a list of numbers is shown: the first one
being the self-intersection of the component, the second one being its multiplicity in the corresponding 
canonical $\QQ$-divisor ($K_Y$ or $K_Z$), and the third one (where applicable) being the ramification index.  
The components $\varepsilon_i^{++}$ are the strict transforms of $e_i^{++}$ by the 5:1 cover, while the 
remaining components $a^{++}$ and $b^{++}$ project onto the double point. Note that the support of $K_Z$
is in the preimage of $R_5$. After blowing down all components, one obtains a smooth 
surface $\hat Z$ with trivial canonical divisor, which is in particular the minimal model of~$\V{\chi_2}{C_2}$.
Using Riemann-Hurwitz once again, one obtains
$$
\chi(\hat Z)=5(\chi(Y)-\chi(R_5))+5=5(14-5\cdot 3)+5=0.
$$

\begin{figure}[ht]
\begin{center}
\begin{tikzpicture}[scale=.5,very thick,>=triangle 45]
\node (ge1++) at (-2.4,0) {$\varepsilon_1^{++}$};
\node[left]  at (ge1++.west) {$(-1,3)$};
\draw (-4,-1)-- (-1,3);
\node (a++) at (0.5,3.6) {$a^{++}$};
\node[below]  at (a++.south) {$(-2,2)$};
\draw (-3,2)-- (3,4);
\node (b++) at (5.3,3.6) {$b^{++}$};
\node[below left=-2]  at (b++.south) {$(-3,1)$};
\draw (2,4)-- (8,2);
\node (ge2++) at (7.2,0) {$\varepsilon_2^{++}$};
\node[right]  at (ge2++.east) {$(-1,2)$};
\draw (6,3)-- (9,-1);
\draw[->] (2.5,1)-- (2.5,-2);
\node (e1++) at (1,-4.5) {$e_1^{++}$};
\node[above left]  at (e1++.west) {$(-2,-\frac{1}{5},1)$};
\draw (-2,-5)-- (3,-3);
\node (e2++) at (4,-4.5) {$e_2^{++}$};
\node[above right]  at (e2++.east) {$(-3,-\frac{2}{5},2)$};
\draw (2,-3)-- (7,-5);
\end{tikzpicture}
\caption{Surface $Y$}
\label{fig-Y}
\end{center}
\end{figure}
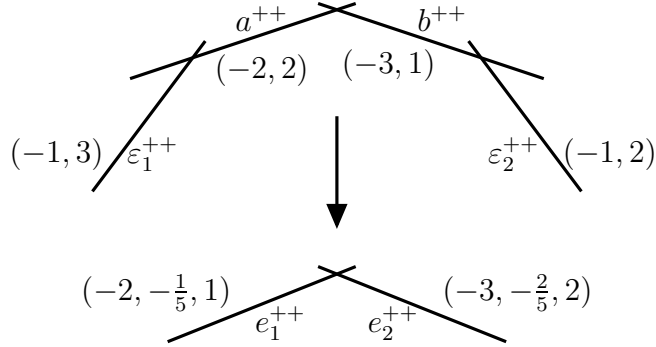

From the Kodaira classification (see~\cite[Table~10]{barth-compact}) the minimal model is a torus
and hence it is an abelian surface.

For part~\ref{thm-degt-2}, note that the degree of the Alexander polynomial of $C_2$ associated with $\chi_2$
(see \cite[section 2.2]{acl-depth}) is $t^4-t^3+t^2-t+1$ (\cite[Theorem 4.5]{AC-prep}). Since
$\dim \Alb(\V{\chi_2}{C_2})=\frac{1}{2}\deg \Delta_{C_2,\chi_2}(t)=2$, the result follows.
\end{proof}

\begin{remark} Note that $\Jac(\AAA_4)$ is a simple abelian variety. This follows
from discussion in \cite{cm} yielding that $\CM$-field in this case
is $\QQ(\zeta_5)$ and explicit description of the $\CM$-type there.
More generally, for the singularity type $x^p+y^q$, where $p,q$ are different prime numbers, 
recall that Arnold-Steenbrink's spectrum provides the $\CM$-type 
for the local Albanese variety (cf. \cite{cm}), whose explicit description is well known. 
One can apply Shimura-Taniyama conditions for primitivity of a $\CM$-type (cf. \cite{Shimura}) 
to verify that the local Albanese variety is simple in this case. 
In particular Theorem~\ref{main}\eqref{simple} can be applied to those plane curve singularities.

In general, however, local Albanese variety has several 
isogeny components. In the case of uni-branched curves
they all are Jacobians of Belyi cyclic covers (cf. \cite{cm})
and hence are the components of Jacobians of Fermat curves.
We refer for additional information regarding these Jacobians
to \cite{koblitz} and \cite{aoki:91}.
\end{remark} 

\begin{thm}\label{ji}
Let $C_3$ be an irreducible curve in $\PP^2$ given by the equation
\begin{equation}
x_0^{2m}+x_1^{2m}+x_2^{2m}-2(x_0^mx_1^m+x_1^mx_2^m+x_2^mx_0^m)=0,
\end{equation}
where $m$ is an odd number, say $m=2\g+1$.
Consider $\V{\chi_3}{C_3}$ the cyclic covering of order~$2m$ of $\PP^2$ ramified along~$C_3$.
Let $D_{\AAA_{2\g}}$ be as above.
Then $\Alb(\V{\chi_3}{C_3})$ is isogenous to $\Jac(D_{\AAA_{2\g}})^3$
and the Albanese dimension of $\V{\chi_3}{C_3}$ is~2.
\end{thm}

\begin{proof}
The pencils of curves
$$\Lambda_i=\{F_{i,[\alpha:\beta]}\mid [\alpha:\beta]\in \PP^1\},$$
(where $F_{i,[\alpha:\beta]}=\{\alpha (x_jx_k)^m + \beta(x_j^m+x_k^m-x_i^m)^2=0\}$
and $\{i,j,k\}=\{0,1,2\}$)
induce orbifold morphisms from $\PP^2$ onto the compact orbifold~$\PP^1_{([1:0],2),([0:1],m)}$.
Since $C_3=F_{i,[-4:1]}$ they also define (by restriction) orbifold morphisms
$\phi_i:\PP^2\setminus C_3\to \PP^1_{2,m,2m}$ defined as
$$[x_0:x_1:x_2]\stackrel{\phi_i\ }{\mapsto} [x_j^mx_k^m:(x_j^m+x_k^m-x_i^m)^2].$$

If one shows that these morphisms are strongly independent, then by Theorem~\ref{main2}, they define a surjective
morphism $\Alb(\V{\chi_3}{C_3})\to \Jac(D_{\AAA_{2\g}})^3$.
Note that $D_{\AAA_{2\g}}$ is a curve of genus $\g$.
Moreover, the Alexander polynomial of $C_3$ associated with $\chi_3$ is the classical Alexander polynomial since
$C_3$ is irreducible, which is $\Delta_{C_3}(t)=\left(\frac{t^{2g}+1}{t+1}\right)^3$ (see~\cite{ji-fundamental}). 
Thus
$$\dim \Alb(\V{\chi_3}{C_3})=\frac{1}{2}\deg \Delta_{C_3}(t)=3\g$$ 
and then $\Alb(\V{\chi_3}{C_3})\sim \Jac(D_{\AAA_{2\g}})^3$ by dimension reasons.

For the last part, consider $(\phi_1\times \phi_2):\PP^2\setminus C_3\to (\PP^1_{2,m,2m})^2$.
Note that the preimage of a generic point is the intersection of two generic members of the pencils
$\Lambda_1$ and $\Lambda_2$ and hence the morphism is finite. The same applies to 
$(\Phi_1\times \Phi_2):\V{\chi_3}{C_3}\to \Sigma^2$. By the standard properties of the Albanese map,
$\alb(\Phi_1\times \Phi_2):\Alb(\V{\chi_3}{C_3})\to \Jac(D_{\AAA_{2\g}})^2$ is surjective. 
Since the Albanese map of $\V{\chi_3}{C_3}$ factors through $\alb(\Phi_1\times \Phi_2)$, the result follows.

It remains to show that the global quotient orbifold pencils $\phi_0$, $\phi_1$, and $\phi_2$ are strongly independent,
in other words, that the morphisms $\Phi_{i,*}:H_1(\V{\chi_3}{C_3})\to H_1(\Sigma)$, $i=0,1,2$, obtained from~\eqref{diagramorbchar3},
\begin{equation}\label{diagramorbchar3}
\begin{tikzpicture}[description/.style={fill=white,inner sep=2pt},baseline=(current bounding box.center)]
\matrix (m) [matrix of math nodes, row sep=2.5em,
column sep=2.5em, text height=1.5ex, text depth=0.25ex]
{\V{\chi_3}{C_3} & \Sigma \\
\PP^2\setminus C_3 & \PP^1_{2,m,2m}\\};
\path[->,>=angle 90](m-1-1) edge node[auto] {$\Phi_i$} (m-1-2);
\path[->,>=angle 90](m-2-1) edge node[auto] {$\phi_i$} (m-2-2);
\path[->,>=angle 90](m-1-1)edge node[auto,swap] {}(m-2-1);
\path[->,>=angle 90](m-1-2)edge node[auto,swap] {}(m-2-2);
\end{tikzpicture}
\end{equation}
are $\ZZ[\mu_{2m}]$-independent ($\mu_{2m}\subset \CC^*$ the cyclic group of $2m$-roots of unity) 
and that $\oplus_{i=0}^2 \Phi_{i,*}:H_1(\V{\chi_3}{C_3})\to H_1(\Sigma)^3$ is surjective
(see Definition~\ref{def-indep})

Note that the base points of the pencils can be described as follows: let $\{i,j,k\}=\{0,1,2\}$ and consider 
$$\Delta_i:=\{x_i=0\}\cap Q_j=\{x_i=0\}\cap Q_k,$$ 
$Q_i:=\{x_j^m+x_k^m-x_i^m=0\}$. The $2m$ base points of $\Lambda_i$ are $\Delta_j\cup \Delta_k$.

In order to understand $\V{\chi_3}{C_3}$ we will first consider a resolution of the base points of the pencil $\Lambda_i$.
This is shown in Figure~\ref{fig-resol}, where
$\tilde \ell_P$ (resp. $\tilde C_3$, and $\tilde Q_i$) represents the strict preimage of $\ell_P$, the axis containing $P$
(resp. $C_3$, and the Fermat curve $Q_i$). The notation $[k]$ next to an irreducible component $E$ indicates the image by 
$\chi_3$ of a meridian $\gamma$ around the irreducible component $E$ as follows:
$$\chi_3(\gamma)=e^{\frac{k}{m}\pi \sqrt{-1}}.$$
Unbranched components, i.e. $[k]=[0]$, are shown in dashed lines.

\begin{figure}[ht]
\begin{center}
\begin{tikzpicture}[scale=.4,>=triangle 45]
\node at (-2.4,-1.3) {$[2 g],E_{g,P}$};
\node at (8.2,0) {$E_{g+2,P},[0]$};
\draw[very thick,dashed] (-1,0)-- (6,0);
\draw[very thick] (-1,-1)-- (1,2);
\node at (3.7,2.5) {$\tilde{C_3},[1]$};
\draw[very thick,->] (2,-1)-- (4,2);
\node at (6,-2) {$E_{g+1,P},[m]$};
\draw[very thick] (4,-1)-- (6,2);
\node at (5,4.6) {$\tilde{Q_i},[0]$};
\draw[very thick,->,dashed]  (6,1)-- (5,4);
\node at (-4.2,3.5) {$[2 g-2],E_{g-1,P}$};
\draw[very thick] (1,1)-- (-1,4);
\node at (-.5,5.5) {$\vdots$};
\node at (-2.8,7.5) {$[4],E_{2,P}$};
\draw[very thick]  (-1,7)-- (1,10);
\node at (1.7,11) {$E_{2,P},[2]$};
\draw[very thick] (1,9)-- (-1,12);
\node at (-1.2,13.5) {$[0],\tilde{\ell}_{P}$};
\draw[very thick,->,dashed] (-1,11)-- (1,14);
\end{tikzpicture}
\caption{}
\label{fig-resol}
\end{center}
\end{figure}

In other words $\V{\chi_3}{C_3}$ is the cyclic covering of order $2m$ ramified along the locus
$$C_3+\sum_{P\in \Delta}\left(2E_{1,P}+4E_{2,P}+\dots +(2g-2)E_{g-1,P}+2gE_{g,P}+mE_{g+1,P}\right),$$
where~$\Delta=\bigcup_{i=0}^2\Delta_i$. To resolve each $\Lambda_i$ it would be enough to blow-up over
$\Delta_j\cup \Delta_k$, but this way the same surface works for the three pencils.

In particular, note that $\V{\chi_3}{C_3}$ will contain curves $\Sigma_P$ which are the cyclic covering of
$E_{g+2,P}$ ramified at 3 points of ramification indices $1$, $m-1$, and~$m$. It is easy to check that the
orders of $\chi_3$ at the meridians of these points are $2m$, $m$, and $2$ respectively. Hence 
$\Sigma_P=\Sigma$ is the curve of genus $g$ which is the Belyi cover $D_{\AAA_{2g}}$ of~$\PP^1_{2,m,2m}$.

Moreover, if $P\in \Delta_k$, then $\Phi_i|_{\Sigma_P}:\Sigma_P\to \Sigma$ and
$\Phi_j|_{\Sigma_P}:\Sigma_P\to \Sigma$ are isomorphisms since $E_{g+2,P}$ in Figure~\ref{fig-resol} is a
dicritical section of $\Lambda_i$ and $\Lambda_j$, whereas $\Phi_k|_{\Sigma_P}:\Sigma_P\to \Sigma$ is a constant map.
This immediately implies the result as follows. Consider three indeterminacy points distributed among the axes,
for instance $P_0:=[0:1:1]$, $P_1:=[1:0:1]$, and~$P_2:=[1:1:0]$. By the previous considerations non-trivial
meridians $\gamma_i\in H_1(\Sigma_{P_i})\cong H_1(\Sigma)$ exist considered as cycles in $H_1(\V{\chi_3}{C_3})$ via
the inclusion and such that
$$\Phi_j(\gamma_i)=\begin{cases} \gamma & \text{ if } i\neq j\\ 0 & \text{ if } j=i,\end{cases}$$
where $\gamma\in H_1(\Sigma)$ is a non-trivial cycle. If $\Phi_{i,*}$ were dependent morphisms, then there
should exist coefficients $\alpha_0,\alpha_1,\alpha_2\in \ZZ[\mu_{2m}]$ such that
$$
\alpha_0 \Phi_{0,*}+\alpha_1 \Phi_{1,*}+\alpha_2 \Phi_{2,*}\equiv 0,
$$
but using the cycle $\gamma_0$ one obtains that $\alpha_1=-\alpha_2$, analogously, using $\gamma_1$ (resp. $\gamma_2$)
one obtains $\alpha_0=-\alpha_2$ (resp. $\alpha_0=-\alpha_1$). Therefore $\alpha_1=\alpha_0=\alpha_2=\alpha$ and $2\alpha=0$
in $\ZZ[\mu_{2m}]$, which implies $\alpha=0$. The fact that the map
$\oplus_{i=0}^2 \Phi_{i,*}:H_1(\V{\chi_3}{C_3})\to H_1(\Sigma)^3$ is surjective follows from the existence
of the dicritical sections $E_{g+2,P_i}$ and the induced isomorphisms $\Phi_j|_{\Sigma_{P_i}}:\Sigma_{P_i}\to \Sigma$
for $j\neq i$ described above.
\end{proof}

\end{section}



\def\cprime{$'$}
\providecommand{\bysame}{\leavevmode\hbox to3em{\hrulefill}\thinspace}
\providecommand{\MR}{\relax\ifhmode\unskip\space\fi MR }
\providecommand{\MRhref}[2]{%
  \href{http://www.ams.org/mathscinet-getitem?mr=#1}{#2}
}
\providecommand{\href}[2]{#2}

\end{document}